\documentclass[10pt,a4paper,reqno]{amsart}
\usepackage{amssymb}
\usepackage{amsfonts}
\usepackage{enumitem}
\newenvironment{enumeratei}{\begin{enumerate}[label=\textup{(\roman*)}, noitemsep, topsep=1.5mm, labelindent=.8em, leftmargin=*, widest=.]}{\end{enumerate}}
\newenvironment{enumeratenum}{\begin{enumerate}[label=\textup{(\arabic*)}, noitemsep, topsep=1.5mm, labelindent=.8em, leftmargin=*, widest=.]}{\end{enumerate}}
\newenvironment{enumeratealph}{\begin{enumerate}[label=\textup{(\alph*)}, noitemsep, topsep=1.5mm, labelindent=.8em, leftmargin=*, widest=.]}{\end{enumerate}}

\usepackage[ruled,titlenumbered,linesnumbered,noend,norelsize]{algorithm2e}

\DontPrintSemicolon
\SetNlSty{footnotesize}{}{}
\IncMargin{0.5em}
\SetNlSkip{1em}
\SetKwIF{If}{ElseIf}{Else}{if}{then}{else if}{else}{endif}
\SetKw{Return}{return}
\SetNlSkip{2em}

\newenvironment{algorithm-hbox}{\hbadness=10000\begin{algorithm}}{\end{algorithm}}

\theoremstyle{plain}
\newtheorem{theorem}{Theorem}

\newtheorem{conjecture}[theorem]{Conjecture}

\newtheorem{lemma}[theorem]{Lemma}

\newtheorem*{claim}{Claim}

\numberwithin{equation}{section}

\newcommand{\ra}{\rightarrow}
\newcommand{\Nat}{\mathbb N}

\newcommand{\NN}{\mathbb N}
\newcommand{\epsi}{\varepsilon}

\renewcommand{\leq}{\leqslant}
\renewcommand{\geq}{\geqslant}

\newcommand{\set}[1]{\{#1\}}
\newcommand{\norm}[1]{{|#1|}}

\DeclareMathOperator{\myFrac}{frac}

\DeclareMathOperator{\PVA}{PCOL}
\DeclareMathOperator{\myRoot}{root}

\DeclareMathOperator{\firstChild}{first-child}
\DeclareMathOperator{\nextChild}{next-child}

\DeclareMathOperator{\depth}{depth}

\DeclareMathSymbol{\desc}{\mathclose}{symbols}{35}
\DeclareMathOperator{\nextV}{\texttt{nextV}}
\DeclareMathOperator{\nextT}{\texttt{next}}

\DeclareMathOperator{\chosen}{chosen}
\DeclareMathOperator{\myPath}{vcolors}
\DeclareMathOperator{\LOG}{Log}
\DeclareMathOperator{\signum}{sgn}

\newcommand{\append}{\cdot}

\newcommand{\Cs}{C^{*}}
\newcommand{\Bs}{B^{*}}

\newcommand{\llen}{\hat{c}}

\begin{document}
\title[Nonrepetitive choice number of trees]{Nonrepetitive choice number of trees}
\author{Jakub Kozik}\thanks{Research of J.\ Kozik and P.\ Micek was supported by the Polish National Science Center within grant 2011/01/D/ST1/04412.}
\address{Theoretical Computer Science Department, Faculty of Mathematics and
Computer Science, Jagiellonian University, 30-387 Krak\'{o}w, Poland}
\email{jkozik@tcs.uj.edu.pl}
\author{Piotr Micek}
\address{Theoretical Computer Science Department, Faculty of Mathematics and
Computer Science, Jagiellonian University, 30-387 Krak\'{o}w, Poland}
\email{piotr.micek@tcs.uj.edu.pl}
\begin{abstract}

A nonrepetitive coloring of a path is a coloring of its vertices such that the sequence of colors along the path does not contain two identical, consecutive blocks. The remarkable construction of Thue asserts that $3$ colors are enough to color nonrepetitively paths of any length. A nonrepetitive coloring of a graph is a coloring of its vertices such that all simple paths are nonrepetitively colored. Assume that each vertex $v$ of a graph $G$ has assigned a set (list) of colors $L_v$. A coloring is chosen from $\set{L_v}_{v\in V(G)}$ if the color of each $v$ belongs to $L_v$. The Thue choice number of $G$, denoted by $\pi_l(G)$, is the minimum $k$ such that for any list assignment $\set{L_v}$ of $G$ with each $\norm{L_v}\geq k$ there is a nonrepetitive coloring of $G$ chosen from $\set{L_v}$. Alon et al.\ (2002) proved that $\pi_l(G)=O(\Delta^2)$ for every graph $G$ with maximum degree at most $\Delta$. We propose an almost linear bound in $\Delta$ for trees, namely for any $\epsi>0$ there is a constant $c$ such that $\pi_l(T)\leq c\Delta^{1+\epsi}$ for every tree $T$ with maximum degree $\Delta$. The only lower bound for trees is given by a recent result of Fiorenzi et al.\ (2011) that for any $\Delta$ there is a tree $T$ such that $\pi_l(T)=\Omega(\frac{\log\Delta}{\log \log \Delta})$. We also show that if one allows repetitions in a coloring but still forbid $3$ identical consecutive blocks of colors on any simple path, then a constant size of the lists allows to color any tree.
\end{abstract}

\maketitle

\section{introduction}

A \emph{repetition} of length $h$ ($h\geqslant 1$) in a sequence is a subsequence of consecutive terms of the form: $x_1\ldots x_h x_1\ldots x_h$. A sequence is \emph{nonrepetitive} if it does not contain a repetition of any length.

In 1906 Thue proved that there exist arbitrarily long nonrepetitive sequences over only $3$ different symbols (see \cite{Ber95,Thu06}). The method discovered by Thue is constructive and uses substitutions over a given set of symbols. Recently a completely different approach to creating long nonrepetitive sequences emerged (see \cite{GKM}). Consider the following naive procedure: generate consecutive
terms of a sequence by choosing symbols at random and every time a repetition occurs, erase the repeated block
and continue. For instance, if the generated sequence is $abcbc$, we must
cancel the last two symbols, which brings us back to $abc$. By a simple counting one can prove that with positive probability the
length of a constructed sequence exceeds any finite bound, provided the
number of symbols is at least $4$. This is slightly weaker than Thue's
result, but the argument seems to be more flexible for adaptations to other settings. This approach leads e.g.\ to a very short proof (see \cite{GKM}) that for every $n\geq1$ and every sequence of sets $L_1,\ldots,L_n$, each of size at least $4$, there is a nonrepetitive sequence $s_1,\ldots,s_n$ where $s_i\in L_i$ (first proved with an enhanced Local Lemma in \cite{GPZ}). The analogous statement for lists of size $3$ remains an exciting open problem. In this paper we make use of the above-mentioned approach to nonrepetitive colorings of trees.

For a given graph $G$  we denote by $V(G)$ the set of vertices of $G$. A coloring function $f:V(G)\to\NN$ is a \emph{nonrepetitive coloring} of $G$ if there is no repetition on the color sequence of any simple path in $G$. The minimum number of colors used in a nonrepetitve coloring of $G$ is called the \emph{Thue number} of $G$ and denoted by $\pi(G)$. The dependence between the Thue number and maximum degree of graphs is already quite well understood.
\begin{theorem}[Alon et al.\ \cite{AGHR02}]\label{thm:Alon}
For any graph $G$ with maximum degree $\Delta$ there is a nonrepetitive coloring of $G$ using at most $16\Delta^2$ colors. Moreover, for every $\Delta>1$ there is a graph with maximum degree $\Delta$ which needs $\Omega\left(\frac{\Delta^2}{\log\Delta}\right)$ colors in any nonrepetitive coloring.
\end{theorem}
The Thue number of any tree is at most $4$ (see \cite{AGHR02}). K{\"u}ndgen and Pelsmajer \cite{KP08} proved that $\pi(G)\leq 12$ for all outerplanar $G$, and $\pi(G)\leq 4^k$ for all graphs $G$ with tree-width at most $k$. Probably the most intriguing question in the area concerns planar graphs.
\begin{conjecture}[Grytczuk 2007 \cite{Gry07b}]
There is a constant such that $\pi(G)\leq c$, for all planar graphs $G$.
\end{conjecture}
\noindent Very recently Dujmovi\'{c} et al.\ \cite{DFJW} showed $\pi(G)=O(\log n)$ for all planar $G$ on $n$ vertices.

Now, we turn to the list-version of nonrepetititve colorings of graphs. This is an analog of the classical graph choosability introduced by Vizing \cite{Viz76} and independently by Erd\"{o}s, Rubin and Taylor \cite{ERT80}. Given a graph $G$ suppose that each $v\in V(G)$ has a preassigned set of colors $L_v$. We call $\set{L_v}_{v\in V(G)}$ a \emph{list assignment} of $G$, or just \emph{lists} of $G$. A coloring $f$ is \emph{chosen from} $\set{L_v}$ if $f(v)\in L_v$ for all $v\in V(G)$. The \emph{Thue choice number} of $G$, denoted by $\pi_l(G)$, is the minimum $k$ such that for any list assignment $\set{L_v}$ of $G$ with each $\norm{L_v}\geq k$ there is a nonrepetitive coloring of $G$ chosen from $\set{L_v}$. The upper bound from Theorem \ref{thm:Alon} works also in the list-setting, i.e., $\pi_l(G)\leq 16\Delta^2$ for all $G$ with maximum degree $\Delta$. As we mentioned $\pi_l(P_n)\leq 4$ for all paths $P_n$ and the problem whether $3$ or $4$ is the right bound remains open. The first significant difference between the Thue number and the Thue choice number has been proved recently for trees.
\begin{theorem}[Fiorenzi et al.\ \cite{FOOZ11}]\label{thm:Thue-choice-for-trees}
For any constant $c$ there is a tree $T$ such that $\pi_l(T)\geq c$.
\end{theorem}
\noindent In fact one can extract from \cite{FOOZ11} that for any $\Delta>1$ there is a tree $T$ with $\pi_l(T)=\Omega(\frac{\log\Delta}{\log \log \Delta})$. We propose two results complementary to Theorem \ref{thm:Thue-choice-for-trees}. First is an improved upper bound for the Thue choice number of trees.
\begin{theorem}\label{thm:1+epsi}
For every $\epsi>0$ there is a constant $c$ such that $\pi_l(T)\leq c\Delta^{1+\epsi}$ for all trees $T$ with maximum degree $\Delta$.
\end{theorem}

A sequence is \emph{of the form} $x^r$ for real $r\geq1$ if it can be divided into $\lceil r\rceil$ blocks where all the blocks but the last are the same, say $x_1\ldots x_n$ for some $n\geq1$, and the last block is the prefix of $x_1\ldots x_n$ of size $\lceil\myFrac(r)\cdot n\rceil$, where $\myFrac(r)$ is the fractional part of $r$. The sequence $x_1\ldots x_n$ repeated in those blocks is also called \textit{the base} of the given sequence. For example any repetition is a sequence of the form $x^2$ and $abcdabcdab$ is of the form $x^{2.5}$ with the base $abcd$. A coloring of a graph $G$ is $x^r$-free for real $r>1$ if there is no sequence of the form $x^r$ among the color sequences of simple paths in $G$. Thus, an $x^2$-free coloring is simply a nonrepetitive coloring while an $x^3$-free coloring satisfies a weaker condition, in particular it allows a coloring to have a repetitions. A consequence of our second result is that for any tree $T$ and lists $\set{L_v}$ each of size $8$ there is an $x^3$-free coloring of $T$ chosen from $\set{L_v}$. This explains somehow the tightness of Theorem \ref{thm:Thue-choice-for-trees}.

\begin{theorem}\label{thm:2+epsi}
For every $\epsi>0$ there is a constant $c$ such that for every tree $T$ and lists $\set{L_v}_{v\in V(T)}$ each of size $c$ there is $x^{2+\epsi}$-free coloring of $T$ chosen from $\set{L_v}$.
\end{theorem}

\section{Proofs}

In both proofs given a tree $T$ we are going to fix an arbitrary vertex for a \emph{root} and denote it by $\myRoot(T)$. For $u,v\in V(T)$ we say that $u$ is a \emph{descendant} of $v$ if the unique simple path from $u$ to $\myRoot(T)$ contains $v$. The set of all descendants of $v$, including $v$, is denoted by $v\desc$. The $\depth(v)$ is the number of vertices on a simple path from $v$ to $\myRoot(T)$. A vertex $u$ is a \emph{child} of $v$ if $u$ is a descendant of $v$ and they are adjacent in $T$. We also pick an arbitrary planar embedding of $T$. This means we fix an ordering of children of every vertex in $T$. 
If $v$ has a child, the first child of $v$ in a determined order is $\firstChild(v)$. 
If $u$ is a child of $v$, but not the last child, then $\nextChild(v,u)$ is the child of $v$ that is next to $u$.

A \emph{vertical} path in a rooted tree is a simple path whose first vertex is a descendant of the last or vice versa. A coloring of a rooted tree $T$ is \emph{vertically $x^r$-free} for real $r>1$ if there is no sequence of the form $x^r$ among the color sequences of vertical paths in $T$.

For any planar embedding of a given rooted tree $T$ and list assignment $\set{L_v}_{v\in V(T)}$, a pair $(f,u)$ is a \emph{partial coloring} if $u\in V(T)$ and $f$ is a partial function from $V(T)$ to $\NN$ defined only for the vertices of $T$ up to $u$ in the preorder traversal of $T$ and $f(v)\in L(v)$, whenever $f(v)$ defined. The set of all partial colorings of a tree $T$ with fixed $\set{L_v}_{v\in V(T)}$ is denoted by $\PVA$.

Following usual convention we define $[n]$ to be $\{1,\ldots,n\}$. For a set of integers $A$ we use $A^+$ to denote the set of finite sequences over $A$ of length at least 1. For $s\in A^+$ and $n \in \NN$ we write $s \append n$ to denote the sequence $s$ with appended element $n$. For a sequence $s= (s_1, \ldots, s_n)$ we put $s_{1..i}= (s_1, \ldots, s_i)$.

Consider a coloring of a rooted tree with an $x^{2+\epsi}$-block on some simple path. Clearly, at least half of the vertices of this path forms a vertical path whose color sequence is of the form $x^{1+\epsi/2}$. Thus, Theorem \ref{thm:2+epsi} is an immediate consequence of the following lemma.
\begin{lemma}\label{lem:ver-free}
For every $\epsi>0$ there is a constant $c=4\cdot\lceil\frac{1}{\epsi}\rceil$ such that every rooted tree is vertically $x^{1+\epsi}$-free colorable from any lists of size $c$.
\end{lemma}
\begin{proof}
For a given $\epsi>0$ put $c=4\cdot\lceil\frac{1}{\epsi}\rceil$. Let $T$ be a rooted tree and $\set{L_v}_{v\in V(T)}$ be the list assignment with each $\norm{L(v)}=c$. In order to get a contradiction, suppose that there is no vertically $x^{1+\epsi}$-free coloring of $T$ chosen from $\set{L_v}$. Fix an arbitrary planar embedding of $T$.

We propose a very naive procedure struggling to build a proper coloring of $T$ from $\set{L_v}$. The procedure maintains $(f,v)$, a partial coloring of $T$ from $\set{L_v}$, with no color sequence of the form $x^{1+\epsi}$ on any vertical paths other than paths going upwards from $v$. To start the procedure we just pick a color for $\myRoot(T)$ from $L(\myRoot(T))$ and all other vertices are uncolored. Every consecutive step of the procedure tries to correct and/or extend the current partial coloring. This is encapsulated by the call of $\nextV((f,v),n)$ function (see Algorithm \ref{alg-next}), where $(f,v)$ is the current partial coloring and $n$ is the hint for the next decision to be made. The call of $\nextV$ checks first whether $(f,v)$ is vertically $x^{1+\epsi}$-free. If not then the colors from vertices in the repeated $\epsi$-part of $x^{1+\epsi}$ occurrence starting from $v$ are erased (as well as colors of all descendants of erased vertices) and the color for the top-most vertex with erased color is set again to be the $n$-th color from its list. If $(f,v)$ is vertically $x^{1+\epsi}$-free, $\nextV((f,v),n)$ tries to extend the partial coloring $(f,v)$ onto the consecutive subtrees of $v$. 
We will keep an invariant that any extension of an input partial coloring $(f,v)$ onto all descendants of $v$ contains a vertical $x^{1+\epsi}$-block. 
We will extend $(f,v)$ onto $u\desc$ for $u$ being consecutive childs of $v$ and if $u$ is the first child of $v$ whose subtree cannot be colored in this way then $\nextV$ sets the color of $u$ to be the $n$-th color from $L(u)$.

\begin{algorithm-hbox}[!ht]
\caption{$\nextV((f,v),n)$}\label{alg-next}
\uIf{\textup{$x^{1+\epsi}$ occurs in $(f,v)$ starting from $v$ on the way to $\myRoot(T)$}}{
 $l=$ the length of the base of $x^{1+\epsi}$ sequence  \label{alg:negative-step-start}\;
 $m  =\lceil l  \cdot \epsi\rceil$\;
 $(v_{l+m},\ldots,v_1) =$ the path starting from $v_{l+m}=v$ going upwards in $T$\;\quad with $f(v_i)=f(v_{l+i})$ for $1\leq i\leq m$\;
 $u\gets v_{l+1}$\;
 erase all values of f in $u\desc$\label{alg:negative-step-end}
}
\uElse{
$u=\firstChild(v)$\label{alg:positive-step-start}\;
\While{\textup{$f$ has a vertically $x^{1+\epsi}$-free extension onto $u\desc$}}{extend $f$ onto $u\desc$ in vertically $x^{1+\epsi}$-free manner\;
$u=\nextChild(v,u)$\label{alg:positive-step-end}\;
}
}
extend $f$ with $\set{u\ra \alpha}$, where $\alpha$ is the $n$-th element of $L(u)$\label{alg:set-a-color}\;
\Return $(f,u)$\;
\end{algorithm-hbox}	

The partial function $\nextV:\PVA\times [c]\to\PVA$ is defined by Algorithm \ref{alg-next}. Note that $\nextV((f,v),n)$ is well-defined for partial colorings $(f,v)$ with
\begin{enumeratei}
\item no color sequence of the form $x^{1+\epsi}$ on a vertical path other than paths going upwards from $v$,\label{item:no-vertical-path} and
\item no $x^{1+\epsi}$-free extension of $(f,v)$ onto $v\desc$.\label{item:no-extension}
\end{enumeratei}
Moreover, if $(f',u)=\nextV((f,v),n)$ then this new partial coloring also satisfies \ref{item:no-vertical-path} and \ref{item:no-extension}. This allows us to iterate the calls of $\nextV$. Note also that vertex $u$ is determined only by $(f,v)$, i.e.\ the first argument of $\nextV$, while $f'(u)$ is simply the $n$-th color in $L(u)$.

Now, we define recursively a function $h:[c]^+\to\PVA$ which captures the idea of our naive procedure trying to color $T$ from $\set{L_v}$. For $s\in [c]^+$, $1\leq n\leq c$ and $\alpha$ being the $n$-th color in $L(\myRoot(T))$ put
\begin{align*}
h(n)&=(\set{\myRoot(T) \rightarrow \alpha},\myRoot(T)),\\
h(s \append n)&=\nextV(h(s),n).
\end{align*}
First of all note that $h(s)$ is well-defined for all $s\in [c]^+$. Indeed, $h(s)$ is explicitly constructed for all $s$ of length $1$ and it trivially satisfies \ref{item:no-vertical-path}, while \ref{item:no-extension} holds as we supposed that there is no vertically $x^{1+\epsi}$-free coloring of $T$ from $\set{L_v}$. Now $h(s \append n)$ is well-defined as $\nextV$ is well-defined for partial colorings satisfying \ref{item:no-vertical-path}-\ref{item:no-extension} and a new partial coloring also satisfies \ref{item:no-vertical-path}-\ref{item:no-extension}.

It is convenient to see $s\in [c]^+$ as a seed driving to a sequence of partial colorings of $T$: $h(s_{1}),h(s_{1..2}),h(s_{1..3}),\ldots,h(s)$. Now, we aim to get a concise description of this sequence. Let $(f_i,v_i)=h(s_{1..i})$ for $1\leq i \leq \norm{s}$. We define $\chosen(s)=(f_1(v_1),\ldots,f_{\norm{s}}(v_{\norm{s}}))$. In other words, $\chosen(s)$ is a sequence of colors set by instruction \ref{alg:set-a-color} of Algorithm \ref{alg-next} in consecutive calls of $\nextV$ on a way to build $h(s)$. 
\begin{claim}
The function $\chosen$ is injective.
\end{claim}
\begin{proof}[Proof of the Claim]
Note that the length of $\chosen(s)$ is equal the length of $s$. To get a contradiction let $s \neq s'$ be the shortest sequences for which $\chosen(s)=\chosen(s')$. Let $n=\norm{s}=\norm{s'}$. By minimality of $s,s'$ we have $s_{1..(n-1)}= s'_{1..(n-1)}$. The first $n-1$ values of $\chosen(s)$ depend only on $s_{1..(n-1)}$, therefore they are the same for both sequences. Moreover, the last values of $\chosen(s)$ and $\chosen(s')$ are picked from the same list. By the construction of the procedure, the list is determined by $h(s_{1..(n-1)})= h(s'_{1..(n-1)})$ or it is just $L(\myRoot(t))$ in case when  $n=1$. Since $s\neq s'$ they must differ on the last coordinate. It means that indices of the last colors in $\chosen(s)$ and $\chosen(s')$ on the list are different, and hence the colors are different.
\end{proof}

Let $s\in [c]^+$, $(f_i,v_i)=h(s_{1..i})$ for all $1\leq i \leq \norm{s}$. 
For $2\leq i \leq \norm{s}$ we denote by $l_i,m_i$ the evaluations of variables $l, m$ in the $(i-1)$-th call to the procedure $\nextV$ (for some calls, they may be undefined). Then $W(s)=(\depth(v_1),\ldots,\depth(v_{\norm{s}}))$ is a \textit{supporting walk} of $s$. The walk contains two kind of steps: positive, when $W(s)_i = W(s)_{i-1}+1$, and negative, when $W(s)_i \leq W(s)_{i-1}$. Positive steps occur when procedure $\nextV$ descends into a subtree, i.e.\ evaluates the case from line \ref{alg:positive-step-start} to \ref{alg:positive-step-end}. Negative steps correspond  to the calls in which repeated part of an $x^{1+\epsi}$-block in the partial assignment is erased (lines \ref{alg:negative-step-start} to \ref{alg:negative-step-end}). Let us suppose that the $i$-th step was negative. Note that just from $W(s)$ we can decode the length of the erased block, i.e.\ the value of  $m_i$. This is exactly  $W(s)_{i-1} - W(s)_i+1$. However, to decode the corresponding value $l_i$ we need some additional information. All we know is that $m_i  =\lceil l_i  \cdot \epsi\rceil$, which leaves $\lceil 1/\epsi \rceil$ possible values for $l_i$. Therefore we annotate every step of $W(s)$ with a number from $\set{0, \ldots, \lceil 1/\epsi \rceil -1}$. The number is meaningful only for negative steps. Formally the annotation function $A:[c]^+ \to \{0, \ldots, \lceil 1/\epsi \rceil-1\}^+$ is defined as follows. For $1\leq i \leq \norm{s}$, 
\[
	A(s)_i = \begin{cases} 
			l_i - \lfloor m_i/\epsi \rfloor & \text{if $i$-th step is  negative}   \\ 
			0 & \text{otherwise.} 
		\end{cases}
\]
Let $s\in [c]^+$ and $(f,v)=h(s)$ then $\myPath(s)$ is the sequence of colors  on the  path from $\myRoot(T)$ to $v$ in a partial coloring $f$. Thus, the last value in $\myPath(s)$ is $f(v)$.

Finally we define a total encoding function $\LOG:[c]^+ \to \Nat^+ \times [1/\epsi]^+ \times \NN^+$ as $\LOG(s)= (W(s), A(s), \myPath(s))$.
\begin{claim}
The function $\LOG$ is injective.
\end{claim}
\begin{proof}
Let $s\in [c]^+$. First we show that $\LOG(s)$ uniquely determines sequences $\myPath(s_{1..i})$ for all $1\leq i\leq \norm{s}$. Recall that $\myPath(s)$ is written explicitly in $\LOG(s)$.

Suppose $\myPath(s_{1..i})$ is already known and now we reconstruct $\myPath(s_{1..(i-1)})$. If the $i$-th step of $W(s)$ is positive, i.e.\ $W(s)_i>W(s)_{i-1}$ then the length of the path from the root to the current vertex increased by 1 in step $i$. Thus, $\myPath(s_{1..(i-1)})$ is exactly the same as $\myPath(s_{1..i})$ but with the last color erased. If the $i$-th step of $W(s)$ is negative then $m_i=W(s)_i-W(s)_{i-1}+1$ is the size of the repeated $\epsi$-block and $l_i=\lfloor m_i/\epsi\rfloor + A(s)_i$ is the size of the base  of a $x^{1+\epsi}$ sequence fixed in this step. Clearly, the last color in $\myPath(s_{1..i})$ is introduced in the $i$-th step and $l_i$ colors before form a base of the $x^{1+\epsi}$ sequence that was retracted. Let $(\alpha_1,\ldots,\alpha_{l_i},\beta)$ be the suffix of $\myPath(s_{1..i})$. Then $\myPath(s_{1..(i-1)})$ is just $\myPath(s_{1..i})$ with the last color, namely $\beta$, erased and sequence $(\alpha_1,\ldots,\alpha_m)$ appended.

Once we have sequences $\myPath(s_{1..i})$ for all $1\leq i \leq \norm{s}$, we may simply read  their last values to reconstruct $\chosen(s)$. Now, the previous Claim assures that $\chosen(s)$ uniquely determines $s$.
\end{proof}

Let us fix  $M\in \NN$. We are going to give a bound for the number of distinct $\LOG(s)$ for $s$ of length $M$ based on the structure of $\LOG(s)$. For $s\in [c]^M$, supporting walk $W(s)$ is a sequence of $M$ positive integers with $W(s)_i-W(s)_{i-1}\leq 1$. Now replace all negative steps $W(s)_i, W(s)_{i+1}$ with a sequence $W(s)_i, W(s)_i+1, W(s)_i, W(s)_i -1, W(s)_i -2 , \ldots, W(s)_{i+1}$. It is easy to see that such an operation is reversible and it results in a sequence of positive integers of size at most $2M$ with all steps in $\set{-1,1}$. The number of such sequences is well-known to be $o(2^{2M})$. The number of possible annotation sequences $A(s)$ is bounded by $\lceil 1/\epsi \rceil^M$. Finally,  $\myPath(s)$ is a sequence of colors which appear on some simple path starting from $\myRoot(T)$ in a final partial coloring $h(s)$. There are $\norm{V(T)}$ simple paths starting from $\myRoot(T)$ and each of them has at most $c^\norm{V(T)}$ possible color assignments.

By the last Claim the number of distinct $\LOG(s)$ for $s\in [c]^M$ is simply $c^M$. On the other hand  the upper bounds from obtained just now we get the following inequality
\[
\displaystyle{c^M\leq o\left(4^M\right)\cdot \left\lceil\frac{1}{\epsi}\right\rceil^M\cdot  \left( \norm{V(T)} \cdot c^{\norm{V(T)}}\right)}.
\]
For $c = 4\lceil\frac{1}{\epsi}\rceil$ this gives a contradiction for sufficiently large $M$.
\end{proof}


\begin{proof}[Proof of Theorem \ref{thm:1+epsi}]

Clearly, it suffices to prove the theorem for small values of $\epsi$. We fix any $\epsi \in (0,1)$, and choose $\delta$ so that it satisfies $1+\epsi=\frac{1+\delta}{1-\delta}$ (note that $\delta<\frac{1}{2}$). We are going to prove a bit stronger statement. 
There is a constant $c$ such that for every rooted tree $T$ with maximum degree at most $\Delta$ and lists $\set{L_v}_{v\in V(T)}$ each of size at least $c \Delta^{1+\epsi}$,  there exists a coloring of $T$ from $\set{L_v}$ with
\begin{enumeratenum}
\item no color sequences of the form $x^2$ on simple paths in $T$,\label{item:no-x2} and
\item no color sequences of the form $x^{1+\delta}$ on vertical paths in $T$.\label{item:no-x1+epsi}
\end{enumeratenum}
Let $c$ be sufficiently large integer ($c\geq 12 \cdot (\lceil \frac{1}{\delta} \rceil+1)$ will do). Let $T$ be a tree and $\set{L_v}_{v\in V(T)}$ be a lists assignment with each $\norm{L(v)}= \llen \geq c\Delta^{1+\epsi}$. In order to get a contradiction, suppose that there is no coloring of $T$ chosen from $\set{L_v}$ with \ref{item:no-x2} and \ref{item:no-x1+epsi} satisfied. Fix an arbitrary planar embedding of $T$.

Like in the proof of Lemma \ref{lem:ver-free} we propose a procedure struggling to accomplish an impossible mission that is to produce a coloring of $T$ from $\set{L_v}$ satisfying \ref{item:no-x2} and \ref{item:no-x1+epsi}. The procedure maintains $(f,v)$ a partial coloring of $T$ from $\set{L_v}$ with the only possible violations of \ref{item:no-x2} and \ref{item:no-x1+epsi} on paths starting at $v$. To start the procedure we just pick a color for $\myRoot(T)$ from $L(\myRoot(T))$ and all other vertices are uncolored. Every consecutive step of the procedure tries to correct and/or extend the current partial coloring. This is encapsulated by the call of $\nextT((f,v),n)$ function (see Algorithm \ref{alg-nextT}), where $(f,v)$ is the current partial coloring and $n$ is the hint for the next decision to be made. The call of $\nextT$ checks first whether $(f,v)$ is vertically $x^{1+\delta}$-free. If not then the colors from vertices in the repeated $\delta$-part of $x^{1+\delta}$ occurrence starting from $v$ are erased (as well as colors of all descendants of erased vertices) and the color for the top-most vertex with color cleared is set again to be the $n$-th color from its list. If $(f,v)$ is vertically $x^{1+\delta}$-free then $\nextT$ checks whether it is $x^2$-free (see lines \ref{alg2:x2-start}-\ref{alg2:x2-end} of Algorithm \ref{alg-nextT}). If there is a path $P$ with a color sequence of the form $x^2$ then it must start at $v$ and $\nextT$ clears the colors along $P$ up to the last vertex which is a predecessor of $v$ or up to the vertex which finishes the repeated block of $x^2$ occurence. Again, the color of the top-most vertex with color cleared is set to be the $n$-th color from its list. Finally, if there is no violation of \ref{item:no-x2} and \ref{item:no-x1+epsi} then $\nextT((f,v),n)$ tries to extend the partial coloring $(f,v)$ onto subtrees rooted at consecutive childs of $v$. 
We will keep an invariant that such an extension of an input partial coloring $(f,v)$ can not be done, and if $u$ is the first child of $v$ whose subtree cannot be colored in this way then $\nextT$ sets the color of $u$ to be the $n$-th color from $L(u)$.

\begin{algorithm-hbox}[!ht]
\caption{$\nextT((f,v),n)$}\label{alg-nextT}
\uIf{\textup{$x^{1+\delta}$ occurs in $(f,v)$ starting at $v$ on the way to $\myRoot(T)$}}{
 $l=$ the length of the base of  $x^{1+\delta}$ sequence\label{alg2:x1+epsi-start}\;
 $m  =\lceil l  \cdot \delta\rceil$\label{alg2:m}\;
 $(v_{l+m},\ldots,v_1) =$ the path starting at $v_{l+m}=v$ going upwards in $T$\;\quad with $f(v_i)=f(v_{l+i})$ for $1\leq i\leq m$\;
 $u\gets v_{l+1}$\;
  erase all values of f in $u\desc$\label{alg2:x1+epsi-end}
}
\uElseIf{\textup{$x^2$ occurs in $(f,v)$ starting at $v$ }}{
 $(v_{2l}, \ldots, v_1) =$ the path starting at $v_{2l}=v$\label{alg2:x2-start}\label{alg2:l}\;
 \quad with $f(v_i)=f(v_{l+i})$ for $1\leq i\leq l$ \;
 $k=$ the least integer $i$ such that $v$ is a descendant of $v_i$\label{alg2:k}\;
 \uIf{$k \leq l$} {$u \gets v_{l+1} $}
 \uElse{$u \gets v_{k+1} $}
  erase all values of f in $u\desc$\label{alg2:x2-end}\label{alg2:x2-erase}
}
\uElse{
$u=\firstChild(v)$\label{alg2:positive-start}\;
\While{\textup{$f$ has an extension onto $u\desc$ satisfying \ref{item:no-x2} and \ref{item:no-x1+epsi}}}{extend $f$ onto $u\desc$ and keep \ref{item:no-x2} and \ref{item:no-x1+epsi} satisfied\label{alg2:extend}\;
$u=\nextChild(v,u)$\label{alg2:positive-end}\;
}
}
extend $f$ with $\set{u\ra \alpha}$, where $\alpha$ is the $n$-th element of $L(u)$\label{alg2:set-a-color}\;
\Return $(f,u)$\;
\end{algorithm-hbox}	

The partial function $\nextT:\PVA\times [\llen]\to\PVA$ is defined by Algorithm \ref{alg-nextT}. Note that $\nextT((f,v),n)$ is well-defined for partial colorings $(f,v)$ with
\begin{enumeratei}
\item no color sequence of the form $x^{1+\delta}$ on a vertical path other than paths going upwards from $v$,\label{item2:no-vertical-path} and
\item no color sequence of the form $x^2$ on simple paths other than paths starting at $v$,\label{item2:no-repetition} and
\item no extension of $(f,v)$ onto $v\desc$ preserving \ref{item:no-x2} and \ref{item:no-x1+epsi}.\label{item2:no-extension}
\end{enumeratei}
Moreover, if $\nextT((f,v),n)$ exists then this new partial coloring also satisfies \ref{item2:no-vertical-path}-\ref{item2:no-extension}. This allows us to iterate the calls of $\nextT$.

Now, we define recursively function $h:[\llen]^+\to\PVA$ which captures the idea of our naive procedure trying to color $T$ from $\set{L_v}$. For $s\in [\llen]^+$, $1\leq n\leq c$ and $\alpha$ being the $n$-th color in $L(\myRoot(T))$ put
\begin{align*}
h(n)&=(\set{\myRoot(T) \rightarrow \alpha},\myRoot(T)),\\
h(s\append n)&=\nextT(h(s),n).
\end{align*}
First of all note that $h(s)$ is well-defined for all $s\in [\llen]^+$. Indeed, $h(s)$ is explicitly constructed for all $s$ of length $1$ and it trivially satisfies \ref{item2:no-vertical-path} and \ref{item2:no-repetition}, while \ref{item2:no-extension} holds as we supposed that there is no coloring of $T$ from $\set{L_v}$ satisfying \ref{item:no-x2} and \ref{item:no-x1+epsi}. Now $h(s\append n)$ is well-defined as $\nextT$ is well-defined for partial colorings satisfying \ref{item2:no-vertical-path}-\ref{item2:no-extension} and a new partial coloring also satisfies \ref{item2:no-vertical-path}-\ref{item2:no-extension}.

Now, for given $s\in[c]^+$ we aim to get a concise description of $h(s_{1})$, $h(s_{1..2})$, $h(s_{1..3}),\ldots,h(s)$. Let $(f_i,v_i)=h(s_{1..i})$ for $1\leq i \leq \norm{s}$. We define $\chosen(s)=(f_1(v_1),\ldots,f_{\norm{s}}(v_{\norm{s}}))$. In other words (and exactly as in the proof of Lemma \ref{lem:ver-free}), $\chosen(s)$ is a sequence of colors set by instruction \ref{alg2:set-a-color} of Algorithm \ref{alg-nextT} in consecutive calls of $\nextT$ on a way to build $h(s)$. 
\begin{claim}
The function $\chosen$ is injective.
\end{claim}
\noindent Note that if $(f',u)=\nextT((f,v),n)$ is defined then vertex $u$ is determined only by $(f,v)$, i.e.\ the first argument of $\nextT$, while $f'(u)$ is simply the $n$-th color in $L(u)$. That is why the proof of the claim above follows exactly the same lines as the proof of the corresponding claim in the proof of Lemma \ref{lem:ver-free}. 

For a partial coloring $(f,v)$ let $\myPath((f,v))$ be the color sequence in $f$ on vertices from $\myRoot(T)$ to $v$. In particular, the last color in $\myPath(s)$ is simply $f(v)$.
\begin{claim}
The function $\myPath$ is injective on partial colorings from the image of $h$.
\end{claim}
\begin{proof}

We are going to prove that for any two partial colorings $(f,v)$, $(f',v')$ from the image of $h$, if $\myPath((f,v))=\myPath((f',v'))$ then  $(f,v)=(f',v')$. The proof goes by induction on the length of $\myPath((f,v))$.

When the length of $\myPath((f,v))$ and so $\myPath((f',v'))$ is $1$ then $v=v'=\myRoot(T)$. Thus, $\myRoot(T)$ is the only vertex colored by $f$ and $f'$, and the statement is trivial.

Suppose that $\norm{\myPath((f,v))}=\norm{\myPath((f',v'))}=n$ and the claim holds for all shorter sequences. Since $(f,v)$ and $(f',v')$ are in the image of $h$ there exist $s,s'\in [\llen]^+$ such that $h(s)=(f,v)$ and $h(s')=(f',v')$. Let $(f_i,v_i)=h(s_{1\ldots i})$ for $1\leq i\leq \norm{s}$ and $(f'_i,v'_i)=h(s'_{1\ldots i})$ for $1\leq i\leq \norm{s'}$. Let $j$ be the least index such that $\depth(v_i)\geq n$ for $j< i\leq \norm{s}$. Analogously, let $j'$ be the least index such that $\depth(v'_i)\geq n$ for $j'< i\leq \norm{s'}$. Now, we need a basic property of Algorithm \ref{alg-nextT} that is, if $(g',u')=\nextT(g,u)$ then the coloring of a path from $\myRoot(T)$ to $u'$, with excluded $u'$, is the same in $g$ and $g'$. This implies that the color sequence from $\myRoot(T)$ to $v_j$ is the same in partial colorings $h(s_{1\ldots i})$ for all $j\leq i \leq \norm{s}$, which is just the prefix of $\myPath((f,v))$ of length $n-1$. Analogously, a color sequence from $\myRoot(T)$ to $v'_{j'}$ is the same in partial colorings $h(s'_{1\ldots i})$ for all $j'\leq i \leq \norm{s'}$, which is just the prefix of $\myPath((f',v'))=\myPath((f,v))$ of length $n-1$. In particular this means that $\myPath((f_j,v_j))=\myPath((f'_{j'},v'_{j'}))$. By the induction hypothesis we get $(f_j,v_j)=(f'_{j'},v'_{j'})$. Now we do know that partial colorings $(f_{j+1},v_{j+1})$ and $(f'_{j'+1},v'_{j'+1})$ are generated by the calls of $\nextT$ with the same first arguments.
Note that Algorithm \ref{alg-nextT} is deterministic (in particular line \ref{alg2:extend}) in a sense that for the same input it always generates the same output. 
Thus, we immediately get that $v_{j+1}=v'_{j'+1}$, say $w=v_{j+1}$, and two partial colorings $(f_{j+1},v_{j+1})$, $(f'_{j+1},v'_{j+1})$ differ at most with the color of $w$. By the definition of $j$ and $j'$, in all the consecutively built partial colorings $(f_i,v_i)$ for $j< i\leq \norm{s}$, $(f'_i,v'_i)$ for $j'<i\leq \norm{s'}$ vertex $w$ is on the path from $\myRoot(T)$ to the current vertex, i.e. $v_i$ or $v'_i$, respectively. Moreover, all these partial colorings differ at most in the subtree of $w$. But the only vertex from $w\desc$ colored in the final colorings (i.e. $(f,v)$ and $(f',v')$) is $w$ itself. Finally, in both of these colorings a vertex $w$ receives the same color which is at the end of $\myPath((f,v))=\myPath((f',v'))$. Thus, $(f,v)=(f',v')$.
\end{proof}

Again (as in the proof of Lemma \ref{lem:ver-free}) we aim to get a concise description of all these partial colorings and then apply a double counting argument. For $s\in [\llen]^+$, let $(f_i,v_i)=h(s_{1\ldots i})$ for all $1\leq i \leq \norm{s}$. For $2\leq i \leq \norm{s}$ we denote by $l_i,k_i$ the valuations of variables $l, k$ in the $(i-1)$-th call to the procedure $\nextT$ (for some calls, they may be undefined). Define $W(s)=(\depth(v_1),\ldots,\depth(v_{\norm{s}}))$ to be a supporting walk of $s$. We distinguish three kind of steps (differences) in $W(s)$
\begin{enumeratealph} 
\item positive, when $W(s)_i=W(s)_{i-1}+1$, i.e.\ no obstruction occures in the $i$-th step and Algorithm \ref{alg-nextT} evaluates lines \ref{alg2:positive-start}-\ref{alg2:positive-end},
\item $x^{1+\delta}$-negative, when $W(s)_i \leq W(s)_{i-1}$ and color sequence of the form $x^{1+\delta}$ is fixed in the $i$-th step; this corresponds to the evaluation of lines \ref{alg2:x1+epsi-start}-\ref{alg2:x1+epsi-end},
\item $x^2$-negative, when $W(s)_i \leq W(s)_{i-1}$ and color sequence of the form $x^2$ is fixed in $i$-th step; this corresponds to the evaluation of lines \ref{alg2:x2-start}-\ref{alg2:x2-end}.
\end{enumeratealph}	
Additionally we put $m_i=W(s)_i-W(s)_{i-1}+1$. For  $x^{1+\delta}$-negative steps, $m_i$ corresponds to the value of variable $m$ in the corresponding call to the procedure $\nextT$.

This time we need three kinds of annotations enriching the information given in $W(s)$. The first is analogous to the one in the proof of Lemma \ref{lem:ver-free} and helps to recover lengths of the base of the $x^{1+\delta}$ sequence  in $x^{1+\delta}$-negative steps. Suppose that the $i$-th step was $x^{1+\delta}$-negative. Note that just from $W(s)$ we can decode the length of the repeated block, i.e.\ the value of variable $m_i$. However, to decode a corresponding value $l_i$ we need some additional information. All we know is that $m_i  =\lceil l_i  \cdot \delta\rceil$, which leaves $\lceil 1/\delta \rceil$ possible values for $l_i$. Therefore we annotate every negative step with a number from $\set{0, \ldots, \lceil 1/\delta \rceil -1}$ and use an extra value for all steps which are not $x^{1+\delta}$-negative. The annotation function $A:[\llen]^+ \to \{-1,0, \ldots, \lceil 1/\delta \rceil-1\}^+$ is defined as follows. For $1\leq i \leq \norm{s}$, 
\[
	A(s)_i = \begin{cases} 
			l_i - \lfloor m_i/\delta \rfloor & \text{if $i$-th step is  $x^{1+\delta}$-negative}   \\ 
			-1 & \text{otherwise.} 
		\end{cases}
\]

The second annotation function will serve to recover basic information concerning the paths whose part was retracted in $x^2$-negative step. Suppose that the $i$-th step is $x^2$-negative. We want to recover the values of $l_i$ and $k_i$  set in lines \ref{alg2:l} and \ref{alg2:k}, which represents the half of length of the path forming a repetition and the position of the tip in this path. Note that $m_i=W(s)_{i-1} - W(s)_i+1$ is equal to $\min(l_i,2l_i - k_i)$. Hence, we need  information what is the difference between $l_i$ and $k_i$. For $1\leq i \leq \norm{s}$ let
\[
	B(s)_i = \begin{cases} 
			l_i-k_i  & \text{if $i$-th step is $x^2$-negative} \\
			\text{whatever} & \text{otherwise.} 
		\end{cases}
\]
To get a more convenient description of function $B$, we make a list of important values of function $B$ and encode it into a sequence over $\set{-1,0,1}$. If the $i$-th step is $x^2$-negative then we convert $B(s)_i$ into a sequence of $0$'s of length $m_i=W(s)_i - W(s)_{i-1}+1$ and if $B(s)_i\neq 0$ we put $\signum(B(s)_i)$ in $\norm{B(s)_i}$-th position. We need to argue here that $\norm{B(s)_i}\leq m_i$. Indeed, as the partial coloring in the $i$-th step has no $x^{1+\delta}$ occurrence we get that $\norm{l_i-k_i}\leq \delta l_i$ and $l_i-m_i\leq \delta l_i$, which give
\[
\norm{B(s)_i}=\norm{l_i-k_i}\leq \delta l_i\leq \frac{\delta}{1-\delta}m_i\leq m_i.
\]
The last inequality holds as $\delta<\frac{1}{2}$. We define $\Bs(s)$ to be the concatenation of the sequences produced for all $x^2$-negative steps.

The third annotation contains the further description of the paths involved in $x^2$-negative steps. Suppose that the $i$-th step is $x^2$-negative and let $P=(v_{2l_i},\ldots,v_1)$ be the path whose color sequence forms a repetition. Already from $W(s)$ and $B^*(s)$ we will recover the size of the path and the value of $k_i$ such that $v_{k_i}$ is the tip of $P$. Now, we want to describe the way in which $P$ goes down in $T$ from $v_{k_i}$ up to $v_1$. Let $n_j$ for $1<j\leq k-1$ be the position of $v_{j-1}$ on the list of children of vertex $v_j$. Then put $C(i)=(n_1,\ldots,n_{k_i-1})$ and $\Cs(s)$ be the concatenation of $C(i)$'s for $i$ being the indices of $x^2$-negative steps.

A total encoding function is defined as $\LOG(s)= (W(s), A(s),\Bs(s),\Cs(s), h(s))$ for $s\in[\llen]^+$. Length of a $\LOG(s)= (W(s), A(s),\Bs(s),\Cs(s), h(s))$ is defined to be the length of $W(S)$, hence $\norm{\LOG(s)}= \norm{s}$. 
 Here comes the key property of $\LOG$ function.
\begin{claim}
The function $\LOG$ is injective.
\end{claim}
\begin{proof}
Take any $L$ from the image of $\LOG$. Suppose that $\norm{L}= n$. Then, there exists $s\in [\llen]^n$ such that $\LOG(s)=L$.  We are going to show that there is only one such $s$. We prove that reconstructing the sequence $\chosen(s)$ from $L$. This will prove the claim as we already know that $\chosen(s)$ is injective.

Let $s'$ be the prefix of $s$ of size $n-1$. In one step of reconstruction we decode from $L$ the last chosen color $\alpha$ and the value of $\LOG(s')$. Then, by simple iteration of this process, we reconstruct the whole $\chosen(s)$. The value of $\alpha$ may be simply read from $h(s)$, which is explicitly given in $\LOG(s)$. In order to get $\LOG(s')$ note that $W(s')$ and $A(s')$ are just the prefixes of $W(s)$ and $A(s)$ of length $\norm{s}-1$. It remains to reconstruct $h(s')$, $\Bs(s')$ and $\Cs(s')$. The way we proceed depends on the type of the last step in $W(s)$, which can be recognized from $W(s)$ itself and $A(s)$. Indeed, if $W(s)_{n}=W(s)_{n-1}+1$, then the last step is positive. Otherwise the value of $A(s)$ indicates which type of negative step we deal with.

\smallskip

\noindent\textbf{Cases 1 and 2.} The last step in $W(s)$ is positive or $x^{1+\delta}$-negative. Then $\Bs(s')=\Bs(s)$, $\Cs(s')=\Cs(s)$. The partial coloring $h(s')$ is reconstructed exactly as in the analogous cases in the proof of Lemma \ref{lem:ver-free}.

\smallskip

\noindent\textbf{Case 3.} The last step in $W(s)$ is $x^2$-negative. Let $P=(v_{2l_n},\ldots,v_1)$ be the path whose color sequence forms a repetition and let $v_{k_n}$ be the tip of $P$. The number of vertices in $P$ with colors erased can be read from $W(s)$ and it is $m_n=W(s)_{n}-W(s)_{n-1}+1$. By the construction of the Algorithm \ref{alg-nextT}  (lines \ref{alg2:x2-start}-\ref{alg2:x2-end}) we have
\[
2l_n-m_n=\max(l_n,k_n).
\]

From the last $m_n$ values of sequence $\Bs(s)$ we can extract the value of $l_n-k_n$. If all these values are zeros then $l_n-k_n=0$. Otherwise exactly one of these $m_n$ values is equal to $1$ or $-1$ and the position of this non-zero value determines $\norm{l_n-k_n}$ while the sign of $l_n-k_n$ is the same as the sign of this non-zero entry. Once we know $d=l_n-k_n$ we can deduce that
\begin{align*}
&l_n=m_n \text{ and } k_n=m_n-d,&&\text{if $d=l_n-k_n\geq0$},\\
&l_n=m_n-d \text{ and } k_n=m_n-2d,&&\text{if $d=l_n-k_n<0$}.
\end{align*}

Let $h(s)=(f,u)$, $h(s')=(f',u')$. As we supposed that call of $\nextV$ generating $h(s)$ from $h(s')$ retracts a repetition on path $P$, we get that $u'=v_{2l_n}$ and $u=v_{2l_n-m_n+1}$. The color of $u=v_{2l_n-m_n+1}$ in $f'$ was erased by line \ref{alg2:x2-erase} and replaced in line \ref{alg2:set-a-color} of Algorithm \ref{alg-nextT}. The colors of $v_{2l_n-m_n+1},\ldots,v_{2l_n}$ were erased from $f'$ and are not visible in $f$ but the colors of $v_1,\ldots,v_{2l_n-m_n}$ remain unchanged. The vertex $v_{2l_n-m_n}$ is clearly the parent of $u$. As we already reconstructed the value of $k_n$, i.e.\ the position of the tip of $P$, we know the vertices of $P$ lying on a path from $v_{2l_n-m_n}$ to $\myRoot(T)$. In particular, we reconstructed the vertex $v_{k_n}$ in $T$. Now, we make use of $\Cs(s)$. The last $k_n-1$ values of $\Cs(s)$ indicates how the path $P$ goes down in $T$ from $v_{k_n}$ up tp $v_1$. This way we reconstructed the position of $(v_{2l_n-m_n},\ldots,v_1)$ in $T$ and we know that their colors are the same in $f$ and $f'$. Once we know the colors of at least first half of the vertices of $P$ (as $m_n\leq l_n$) and as the color sequence of vertices from $P$ forms a repetition in $f'$ we may deduce the colors of $v_{2l_n-m_n+1},\ldots,v_{2l_n}$.

Putting all together we finally reconstruct $\myPath(h(s'))$ which is the sequence of colors in $f'$ from $\myRoot(T)$ down to $u'=v_{2l_n}$. Indeed, the colors from $\myRoot(T)$ down to $v_{2l_n-m_n}$ are the same in $f'$ and $f$, while the colors from $v_{2l_n-m_n+1}$ to $v_{2l_n}$ has just been reconstructed. Now, recall that the function $\myPath$ is injective on the partial colorings from the image of $h$ which means that we can reconstruct from $\myPath(h(s'))$ a partial coloring $h(s')$ itself.
\end{proof}

We are going to bound the number of distinct $\LOG(s)$ for $s$ of length $M$.
For every $s\in [\llen]^M$ we have $\LOG(s)=(W(s),A(s),\Bs(s),\Cs(s),h(s))$. Just like before, the number of integer walks $W(s)$ of length $M$ is $o(4^{M})$. The number of possible annotation sequences $A(s)$ is bounded by $(\lceil 1/\delta \rceil+1)^M$. The annotation $\Bs(s)$ is a sequence of numbers $\set{-1,0,1}$ of length $\sum_i m_i$, where $i$ goes over all the indices of $x^{1+\delta}$-negative steps. Clearly, $\sum_i m_i \leq M$ and so the number of distinct $\Bs(s)$ is bounded by $3^M$. The annotation $\Cs(s)$ is the concatenation of sequences over $\set{1,\ldots,\Delta-1}$. The length of $\Cs(s)$ is equal to $\sum_i k_i$, where the sum goes over the set $I_{x^2}$ of all the indices of $x^2$-negative steps. Clearly,
\[
\sum_{i\in I_{x^2}} k_i \leq  \sum_{i\in I_{x^2}} (1+\delta) l_i\leq \sum_{i=1}^M \frac{1+\delta}{1-\delta}m_i\leq \frac{1+\delta}{1-\delta}M=(1+\epsi)M, 
\]
By the last Claim the number of distinct $\LOG(s)$ for $s\in [\llen]^M$ is simply $\llen^M \geq (c\Delta)^{(1+\epsi)M}$. On the other hand we just obtained an independent upper bound and altogether we get the following inequality
\[
\displaystyle{ (c\Delta)^{(1+\epsi)M} \leq o\left(4^M\right)\cdot \left(\left\lceil\frac{1}{\delta}\right\rceil+1\right)^M \cdot 3^M \cdot \Delta^{(1+\epsi)M} \cdot \left(\norm{V(T)} c^{\norm{V(T)}}\right)}.
\]
For $c \geq 12 \cdot (\lceil \frac{1}{\delta} \rceil+1)$ and sufficiently large $M$ we get a contradiction.
\end{proof}

\bibliographystyle{plain}
\bibliography{nonrepetitive-trees}
\end{document}